\documentclass[12pt,a4paper]{amsart}
\usepackage{amsfonts}
\usepackage{amsthm}
\usepackage{amsmath}
\usepackage{amscd}
\usepackage[latin2]{inputenc}
\usepackage{t1enc}
\usepackage[mathscr]{eucal}
\usepackage{indentfirst}
\usepackage{graphicx}
\usepackage{graphics}
\usepackage{pict2e}

\numberwithin{equation}{section}

\theoremstyle{plain}
\newtheorem{Th}{Theorem}%[section]
\newtheorem{Lemma}[Th]{Lemma}
\newtheorem{Cor}[Th]{Corollary}

\newtheorem{Pro}[Th]{Proposition}
 \theoremstyle{definition}

\newtheorem{?}[Th]{Problem}

\newcommand{\la}{\lambda}

\DeclareMathOperator{\diag}{diag}

\newcommand{\C}{\mathbb{C}}

\begin{document}

\title[Identities for matrices over the Grassmann algebra]
{Polynomial identities for matrices over the Grassmann algebra}

\author[P. E. Frenkel]{P\'eter E. Frenkel}

\address{E\"{o}tv\"{o}s Lor\'{a}nd University \\ Department of Algebra and Number Theory
 \\ H-1117 Budapest
 \\ P\'{a}zm\'{a}ny P\'{e}ter s\'{e}t\'{a}ny 1/C \\ Hungary \\ and R\'enyi Institute of Mathematics, Hungarian Academy of Sciences \\ 13-15 Re\'altanoda utca \\ H-1053 Budapest}
\email{frenkelp@cs.elte.hu}

\thanks{Research partially supported by ERC Consolidator Grant 648017 and by Hungarian National Foundation for Scientific Research (OTKA),  grants no.\ K109684 and K104206.}

 \subjclass[2010]{}


\begin{abstract}We determine minimal Cayley--Hamilton and Capelli identities for matrices over a Grassmann algebra of finite rank. For minimal standard identities, we give lower and upper bounds on the degree.  These results  improve on upper bounds given by  L.\ M\'arki, J.\ Meyer, J.\ Szigeti, and L.\ van Wyk in a  recent paper.
\end{abstract}

\maketitle

\section{Notations}
Let $R$ be a commutative ring with 1.
For $m\ge 0$, consider the Grassmann algebra \[E^m=R\langle v_1,\dots, v_m\rangle/(v_k^2, v_iv_j+v_jv_i\;|\; 1\le k\le m, 1\le i<j\le m)\] of rank $m$.
It is a  graded $R$-algebra (each $v_i$ has degree 1).  We write $E^m_i$ for the degree $i$ component, so $$E^m=\bigoplus_{i=0}^mE^m_i, \qquad E_0^m=R, \quad E_m^m=Rv_1\cdots v_m.$$
We write $$E^m_{\ge r}=\bigoplus_{i=r}^m E_i^m.$$
Let $M_nX$ be the set of $n$-square matrices with entries in the set $X$.

\section{Cayley--Hamilton identity}
L.\ M\'arki, J.\ Meyer, J.\ Szigeti, and L.\ van Wyk have shown \cite[3.4 Theorem]{MMSzW} that
if the base ring $R$ is a  field of characteristic zero and $m\ge 2$, then any
element of $M_nE^m$ satisfies a monic polynomial of degree $n\cdot 2^{m-1}$
over $R$. This was achieved by constructing  a CT-embedding of $E^m$ into a $2^{m-1}$-square matrix algebra over a suitable commutative $R$-algebra --- an interesting result in its own right. In the present paper, we allow $R$ to be any commutative ring with 1 and reduce the degree of the monic polynomial from $n\cdot 2^{m-1}$ to $n\cdot \left(\lceil m/2\rceil+1\right)$, which turns out to be least possible in general. Moreover, a  suitable polynomial of this degree is given explicitly. For the proof, we do not use CT-embeddings. Instead, we directly exploit the nilpotency and supercommutativity properties of the Grassmann algebra.

We now set up notation that will be used throughout this section.
Let $A\in M_nE^m$. We decompose $A$ into its homogeneous components:
$$A=\sum_{i=0}^mA_i, \qquad A_i\in M_nE^m_i.$$  Let $$f(x)=\det (xI-A_0)\in
R[x]$$ be the (monic) characteristic polynomial of $A_0\in M_nR$.  The main
result of this paper is

\begin{Th}\label{CH}
For any $A\in M_nE^m$, we have \[f(A)^{\lceil m/2\rceil+1}=0.\]
\end{Th}
For $m=0$, this recovers the Cayley--Hamilton Theorem.

As a first step towards the proof, we decompose $B=f(A)\in M_nE^m$ into its homogeneous components:
$$B=\sum_{i=0}^mB_i, \qquad B_i\in M_nE^m_i.$$
Before attacking Theorem~\ref{CH} in its full generality, we treat a special
case.

\begin{Lemma}\label{lemma}
If the degree zero component $A_0$ of $A\in M_nE^m$ is a diagonal matrix $A_0=\diag(\la_i)_{i=1}^n$ with distinct diagonal elements $\la_i\in R$, and $A_1=(v_{ij})$ with $v_{ij}\in E^m_1$, then \begin{enumerate}

\item[(0)] \[B_0=0,\]

\item[(1)] \[B_1=\diag (f'(\la_i)v_{ii})_{i=1}^n,\]

\item[(2)] \[B_1^2=0.\]
%\item[(2+)] $$(B_2)_{rr}=\sum_{i=1}^n\frac{f'(\la_r)}{\la_r-\la_i}v_{ri}v_{ir}$$ for $r=1,\dots, n$,

\item[(3)] $$(B_2)_{rs}=\frac1{\la_r-\la_s}\left({f'(\la_r)}v_{rr}+{f'(\la_s)}v_{ss}\right)v_{rs}$$ for $r\ne s$.

\item[(4)] If $B_2^+$ and $B_2^-$ denote the diagonal and  off-diagonal part of $B_2$ respectively, then $B_1$ commutes with $B_2^+$ but anticommutes with $B_2^-$.

\end{enumerate}
\end{Lemma}

\begin{proof}
We have $$f(x)=\prod_{i=1}^n(x-\la_i),$$
whence \begin{equation}\label{B}B=f(A)=\prod_{i=1}^n(A-\la_iI).\end{equation}

%\begin{enumerate}
%\item[(0)]

  (0) We have $B_0=f(A_0)$, which is zero by the Cayley--Hamilton Theorem, or, if you prefer, by the  trivial computation $$f(A_0)=\prod_{i=1}^n(A_0-\la_iI)=\diag\left(\prod_{i=1}^n(\la_j-\la_i)\right)_{j=1}^n=0.$$

%\item[(1)]

(1)  The factors in \eqref{B} commute, so, for any indices $r\ne s$, we
  have \begin{equation}\label{offdiag}B=(A-\la_rI)C(A-\la_sI)\end{equation}
 for some $C\in M_nE^m$. In the first factor,
  the $r$-th row has no degree zero component. In the last factor, the $s$-th
  column has no degree zero component. Thus, the $(r,s)$ entry in $B$ has no degree
1 component.

The degree 1  component of the $(r,r)$ entry in $B$ arises by taking  a  degree 1
component from the $r$-th row of $A-\la_rI$ and degree zero components from
all other $A-\la_iI$. But the degree zero components of these matrices are
diagonal, so the only possibility is to use the $(r,r)$ entry from each
factor. The result is $$v_{rr}\prod_{i\ne r}(\la_r-\la_i),$$ as claimed.

%\item[(2)]

(2) This is clear because $B_1$ is diagonal and homogeneous of degree 1.

%\item[(3)]

(3)  Using formulas \eqref{B} and \eqref{offdiag}, we obtain \[(B_2)_{rs}=\sum_{j=1}^n v_{rj}v_{js}\prod_{i\ne
  r,s}(\la_j-\la_i)
=v_{rr}v_{rs}\frac{f'(\la_r)}{\la_r-\la_s}+v_{rs}v_{ss}\frac{f'(\la_s)}{\la_s-\la_r},\]
    which yields the result.

%\item[(4)]

(4)  The first statement is clear because $B_1$ and $B_2^+$ are diagonal
  and $B_2^+$ is homogeneous of degree 2 --- note that in the Grassmann algebra, homogeneous elements of even degree are central.

For the second statement, observe that $B_1$ is diagonal and $B_2^-$ is
off-diagonal, so their product, in either order, is off-diagonal. Moreover,
for $r\ne s$, the $(r,s)$ entry in the product is
 $$(B_1B_2^-)_{rs}=(B_1)_{rr}(B_2)_{rs}=f'(\la_r)v_{rr}\frac1{\la_r-\la_s}{f'(\la_s)}v_{ss}v_{rs}$$
for one order and is
$$(B_2^-B_1)_{rs}=(B_2)_{rs}(B_1)_{ss}=\frac1{\la_r-\la_s}{f'(\la_r)}v_{rr}v_{rs}f'(\la_s)v_{ss}$$
for the other order. These add up to zero as claimed.
%\end{enumerate}
\end{proof}

\bigskip\noindent\it Proof of Theorem~\ref{CH}. \rm
The coefficients of the polynomial $f(x)$ are polynomials with
integer coefficients in the entries of $A_0$. Hence, the coordinates in the
natural $R$-basis \begin{equation}\label{basis}\{v_{i_1}\cdots v_{i_k} | i_1<\dots <i_k\}\end{equation} of the
entries of the matrix $f(A)^{\lceil m/2\rceil+1}$ are polynomials with integer coefficients in the
coordinates of the entries of $A$. The theorem is equivalent to the statement
that these $n^22^m$ polynomials are all identically zero. Thus, we may assume
that $R=\C$.
We may assume that $A_0$ has $n$ distinct eigenvalues, since such matrices are
dense in $M_n(\C)$.Then  $A_0$ is diagonalizable by an invertible complex
matrix $P$. Since the conjugation by $P$ is an automorphism of $M_nE^m$ as a
graded algebra over $\C$, we may assume that $P=I$, i.e., $A_0$ is diagonal
with distinct diagonal entries.
Then, by Lemma~\ref{lemma}, we have
\[B=\sum_{i=1}^m B_i, \] where $ B_i\in M_nE^m_i$,  $B_1^2=0$, $B_2=B_2^++B_2^-$, and
$B_1B_2^\pm=\pm B_2^\pm B_1$.
If $B^k\ne 0$ for an exponent $k$, then there is a nonzero
product \[B_{i_1}^\bullet\cdots B_{i_k}^\bullet,\] where $i_1,\dots,
i_k\in\{1,\dots, m\}$, and $B_i^\bullet$ means $B_2^\pm$ if $i=2$ and means
$B_i$ otherwise. But then, in the sequence $i_1$, \dots, $i_k$, any
two 1's are separated by at least one $i\ge 3$, whence
\[m\ge i_1+\dots+i_k\ge 2k-1,\]  so $k\le \lfloor (m+1)/2\rfloor=\lceil m/2\rceil.$\hfill\qed
%\end{proofof}

\bigskip

We now show that the degree of the polynomial in  Theorem~\ref{CH} cannot be
reduced.

\begin{Pro}  Let $R$ be a field of characteristic either 0 or  a prime $p>\lceil
  m/2\rceil$. Let $\la_1,\dots, \la_n\in R$ be distinct elements
  and  \[v=v_1v_2+v_3v_4+\dots+ v_{2\lfloor m/2\rfloor -1}v_{2\lfloor
    m/2\rfloor } \{+v_m\}\in E^m\]  (the last term appears only if
  $m$ is odd).

Let  \[A=\diag(\la_i+v)_{i=1}^n\in M_nE^m,\] so that
$A_0=\diag(\la_i)_{i=1}^n$. Then the characteristic polynomial of $A_0$
is \[f(x)=\prod_{i=1}^n(x-\la_i)\] and the minimal polynomial of $A$ over $R$
is \[f(x)^{\lceil m/2\rceil+1}.\]
\end{Pro}

\begin{proof}  Observe that  \[ v^{\lceil m/2\rceil} =\lceil m/2\rceil
  !v_1\dots v_m\ne 0.\]  Thus, the polynomial

\[g_i(x)=\frac{f(x)^{\lceil m/2\rceil+1}}{x-\la_i}=(x-\la_i)^{\lceil
  m/2\rceil}\prod_{j\ne i}(x-\la_j)^{\lceil
  m/2\rceil+1}\] does not vanish at
$A$. Indeed, the $(i,i)$-entry of $g_i(A)$ is \[v^{\lceil
  m/2\rceil}\prod_{j\ne i} (\la_i-\la_j+v)^{\lceil
  m/2\rceil+1}=v^{\lceil
  m/2\rceil}f'(\la_i)^{\lceil
  m/2\rceil+1}\ne 0.\]
\end{proof}

\section{Capelli identity}
Recall \cite[Definition 1.5.3]{GZ} that the Capelli polynomial $d_k$ is defined by the formula
\[d_k(x_1,\dots, x_k; y_0,\dots, y_k)=\sum_{\pi\in\mathfrak S_k} (-1)^\pi
y_0x_{\pi(1)}y_1x_{\pi (2)}\cdots y_{k-1}x_{\pi (k)}y_k.\]
We say that the Capelli identity of $x$-degree $k$ holds in a ring $\mathfrak
A$ if the above expression is 0 for all $x_1,\dots, x_k, y_0,\dots,
y_k\in\mathfrak A$.
It is trivial that the Capelli identity of $x$-degree $k$
implies the Capelli identity  of $x$-degree $k+1$.

It is well known that the ring of $n$-square matrices over a commutative ring  satisfies
the Capelli identity of $x$-degree $n^2+1$ (because the Capelli polynomial is alternating in the variables $x_i$), but does not satisfy the Capelli
identity of $x$-degree $n^2$ if the base ring has $1\ne 0$ (because we may choose the $x_i$ to be the usual matrix units in some order, and choose the $y_i$ to be suitable matrix units such that exactly one term in $d_{n^2}$ is nonzero). We now wish to
generalize this to matrices over the Grassmann algebra $E^m$. We shall need
the following  lemma.

\begin{Lemma}\label{perm}
Let $a_1, \dots,  a_k%\in \{1,v_1,\dots, v_m\}\subset E^m
$ be elements of a ring. Suppose that $[k]=\{1,\dots, k\}=M\cup N$ is a disjoint union. Suppose that  $a_i$ and $a_j$ anticommute for distinct $i,j\in M$, but commute otherwise.%each of $v_1$, \dots, $v_m$ occurs exactly once among the $a_i$%, and all other $a_i$ are 1

Let $\mathcal P$ be a partition of $[k]$ into $|N|$ classes, each class containing exactly one element of $N$. %$i$ such that $a_i=1$.
 Let $\mathfrak S\subseteq \mathfrak S_k$ be the Young subgroup corresponding to $\mathcal P$ (i.e., $\mathfrak S$ is the group of permutations leaving each class invariant).

\bigskip

(a) If $|M|$ is odd, then
\begin{equation}\label{odd}\sum_{\pi\in\mathfrak S}(-1)^\pi a_{\pi(1)}\cdots a_{\pi(k)}=0.\end{equation}

\bigskip

(b) If %($|M|$ is even and)
 $\mathcal P$ consists of intervals of odd cardinalities $m_1+1$, \dots, $m_{|N|}+1$ respectively, and $N$ consists of the leftmost elements of these intervals,  then
\begin{equation}\label{even}\sum_{\pi\in\mathfrak S}(-1)^\pi a_{\pi(1)}\cdots a_{\pi(k)}= m_1!\cdots m_{|N|}!a_1\cdots a_k.\end{equation} % if $|M|=m$ is even, $k=2m$, and $$M=\{3,4,7,8,\dots, k-1,k\},$$ $$\mathcal P=\{\{1,3\}, \{2,4\}, \{5,7\}, \{6,8\},\dots, \{k-3,k-1\}, \{k-2,k\}\}.$$
\end{Lemma}

\begin{proof} (a) We use induction on $m=|M|$. For $m=1$, the group $\mathfrak S$ has two elements of distinct sign, and all $a_i$ commute, so \eqref{odd} holds.

Let $m\ge 3$ be odd. Suppose that the claim is true for $m-2$. Let us prove it for
$m$.

Consider the special case when $\mathcal P$ is a partition into intervals.
Then the left hand side of \eqref{odd} can be written as a product. For each interval $ I\in\mathcal P$, we get a factor of the form
\begin{equation}\label{interval}\sum_{\pi\in\mathfrak S_I}(-1)^\pi \prod_{i\in I} a_{\pi(i)}.\end{equation}
Since $m$ is odd, we can choose an  interval $I\in\mathcal P$ that has an even number of elements. There is  a unique $i\in I\cap N$.
The terms in \eqref{interval} where $\pi^{-1}(i)$ is even can be paired off with those where
$\pi^{-1}(i)$ is odd. This can be done so that in each pair $\pi$ has different signs but $\prod_{i\in I} a_{\pi(i)}$ is the same, so the sum within each pair is zero. This proves the special case.

To finish the proof, it suffices to prove the following. If the lemma is true for a  sequence $a_1$, \dots, $a_k$ and a partition $\mathcal P$, and $i-1$ and $i$ are in distinct classes $S$ and $T$ of $\mathcal P$ respectively, then the lemma remains true for  $a_i'=a_{i-1}$, $a_{i-1}'=a_i$, $M'=M\Delta \{i-1,i\}$,  $N'=N\Delta \{i-1,i\}$,  $S'=(S-\{i-1\})\cup \{i\}$, $T'=(T-\{i\})\cup \{i-1\}$ (all other data remain unchanged).

To prove this, we examine the change made in the left hand side of  \eqref{odd}. The terms  remain the same up to order and sign.  The terms that  change sign are exactly those where  ${\pi (i-1)}$ and ${\pi (i)}$  both come from the set $M$.  It suffices to prove that these terms sum to zero. This is true even if   ${\pi (i-1)}$ and ${\pi (i)}$ are fixed elements of $M$, due to the induction hypothesis.
 \bigskip

(b) We may assume that $|N|=1$. Then the claim is trivial.
\end{proof}

\begin{Th}\label{Capelli} The ring $M_nE^m$ satisfies the Capelli identity of
  $x$-degree $k=n^2+2\lfloor m/2\rfloor+1$.
\end{Th}

\begin{proof}
Let $A_1,\dots, A_k, B_0,\dots, B_k\in M_nE^m$. We prove
that \begin{equation}\label{capelli} d_k(A_1,\dots, A_k; B_0,\dots,
  B_k)=0.\end{equation} By multilinearity, we may assume that each $A_i$ and
 each $B_i$
has only one nonzero entry, which is an element of the standard $R$-basis \eqref{basis} of $E^m$.
%\[v_{I_i}=\bigwedge_{j\in I_i}v_j.\]
Moreover, we may assume that the degrees of these $2k+1$ basis elements sum to at most $m$.
%\[v_{I_1}\cdots v_{I_k}\ne 0,\] i.e., the sets
%$I_1, \dots, I_k\subset [m]$ are pairwise disjoint.
 Then  at most $m$ of these degrees are nonzero. I.e., at least $2k+1-m$
of these $2k+1$  basis elements are 1. At least $k-m$ of these 1's come from the matrices $A_i$.

%m
%are non-empty, so at least $k-m$ are empty. The corresponding matrices $A_i$
%have an entry 1 and all other entries zero.

If $m$ is even, then $k-m=n^2+1$,
so, by the pigeonhole principle, there exist indices $i\ne i'$ such that
$A_i=A_{i'}$, whence \eqref{capelli} holds.

If $m$ is odd, then $k-m=n^2$. We may assume that $$A_{1}, \dots, A_{n^2}$$ is
the standard $R$-basis of $M_nR$, while $$A_{n^2+i}=v_iA_{j_i},$$ where $1\le j_i\le n^2$ for $i=1,\dots,
m$. We may also assume that each $B_i$ comes from the  standard $R$-basis of $M_nR$. %with $1\le r_i\le n^2$ for all $0\le i\le k$. %are homogeneous of degree zero, i.e.,
%$B_i\in M_nR$.
 The claim now follows from Lemma~\ref{perm}(a), applied to the nonzero
 entries of the matrices $A_i$.
 \end{proof}

\begin{Pro}
The ring $M_nE^m$ does not satisfy the Capelli identity of
  $x$-degree $k=n^2+2\lfloor m/2\rfloor$ if the base ring $R$ is a field of characteristic either zero or  a prime $p>2\left\lceil\lfloor m/2\rfloor /n^2\right\rceil$.
\end{Pro}

\begin{proof} Let us write $2\lfloor m/2\rfloor$ as a  sum of $n^2$ even numbers that are smaller than $p$ if the characteristic is $p>0$. Let these even numbers be $m_1$, \dots, $m_{n^2}$. Let $A_1$, \dots, $A_{n^2}$ be the standard basis of $M_nR$. For each $r$, consider $m_r$ matrices of the form $v_iA_r$, chosen so that each index $i=1, \dots, 2\lfloor m/2\rfloor$ is used exactly once. Let us insert the chosen $m_r$ multiples of $A_r$ immediately after $A_r$ into the sequence $A_1$, \dots, $A_{n^2}$.  This gives us  a sequence $C_1$, \dots, $C_k$. Now let $B_0$, \dots, $B_k$ be elements from the standard basis of $M_nR$ with the property that $B_0C_1B_1\cdots C_kB_k\ne 0$. Then \begin{equation}\label{fakt}d_k(C_1,\dots, C_k; B_0, \dots, B_k)=B_0C_1B_1\cdots C_kB_k\prod_{r=1}^{n^2}m_r!
\end{equation} by Lemma~\ref{perm}(b), applied to the case where $a_i$ is the unique
  nonzero entry of the matrix $C_i$  ($i=1,\dots, k$). The right hand side of \eqref{fakt} is nonzero because $m_r<p$ if the characteristic is $p>0$.
\end{proof}

\section{Standard identity}
The standard polynomial $s_k$ is defined by the formula
\[s_k(x_1,\dots, x_k)=\sum_{\pi\in\mathfrak S_k} (-1)^\pi
x_{\pi(1)}x_{\pi (2)}\cdots x_{\pi (k)}.\]
We say that the standard identity of degree $k$ holds in a ring $\mathfrak
A$ if the above expression is 0 for all $x_1,\dots, x_k\in\mathfrak A$.
It is trivial that the standard identity of degree $k$
implies the standard identity  of degree $k+1$. When $\mathfrak A\ni 1$ and $k$ is even, the converse implication holds as well because $$s_k(x_1, ..., x_k) =s_{k+1}(1, x_1, ..., x_k).$$ 

Also, the Capelli identity of $x$-degree $k$ implies the standard identity of degree $2\lfloor k/2\rfloor$ if $\mathfrak A\ni 1$. Indeed, we may substitute 1 for each $y_i$ in the Capelli identity to get the standard identity of degree $k$, and then we can use the previous remark.

The celebrated Amitsur--Levitzki theorem \cite{AL1}, see e.g.\ also \cite{GZ}, states  that the ring of $n$-square matrices over a commutative ring  satisfies
the standard identity of degree $2n$. An easy example shows that it  does not satisfy the standard
identity of degree $2n-1$ if the base ring has $1\ne 0$. We now wish to
generalize this to matrices over the Grassmann algebra $E^m$.

L.\ M\'arki, J.\ Meyer, J.\ Szigeti, and L.\ van Wyk  \cite[3.7 Theorem]{MMSzW}
used an embedding into a  matrix algebra over a commutative ring and invoked
the Amitsur--Levitzki Theorem to show that for $m\ge 1$, the standard identity
of degree $2^mn$ holds  in $M_nE^m$.  They also invoked a  very general theorem of
M.\ Domokos \cite [Theorem 5.5]{D}
to  show that the standard identity of degree $(m+1)n^2+1$ holds in $M_nE^m$
\cite[3.8 Remark]{MMSzW}. We now show that these degree bounds can be
substantially reduced. For the latter one, this is already clear from Theorem~\ref{Capelli}, which yields 
\begin{Cor}\label{6} The standard identity of degree $$2\left(\left\lfloor\frac{n^2+1}2\right\rfloor+\left\lfloor \frac m2\right\rfloor\right)$$ holds in $M_nE^m$.
\end{Cor}

An improvement of the  degree bound 
$2^mn$ is given by 

\begin{Pro}\label{st}
The standard identity of degree $k=2n(\lfloor m/2\rfloor+1)$ holds in $M_nE^m$.
\end{Pro}

\begin{proof}
We prove the stronger identity \[s_{2n}(A_1,\dots,
A_{2n})s_{2n}(A_{2n+1},\dots, A_{4n})\cdots s_{2n}(A_{k-2n+1},\dots, A_k)=0\]
for all $A_1,\dots, A_k\in M_nE^m$. It suffices to prove that each of the
$\lfloor m/2\rfloor +1$ factors is contained in  $M_nE^m_{\ge 2}$. In fact, it
suffices to prove this for the first factor. Observe that the
ring $$E^m/(v_2,\dots, v_m)\simeq R[v_1]/(v_1^2)$$ is commutative.  Thus, by the
Amitsur--Levitzki Theorem, $n$-square matrices over this ring satisfy the
standard identity of degree $2n$. Thus, each entry in the matrix
$s_{2n}(A_1,\dots, A_{2n})$ is contained in the ideal $(v_2,\dots, v_m)$;
moreover, by the same argument, it is contained
in  \[\bigcap_{i=1}^m(v_j|j\ne i)=E^m_{\ge 2},\] as claimed.
\end{proof}

Note that for $m=0$ or $m=1$, the ring $M_nE^m$ is commutative and
Proposition~\ref{st} reduces to the Amitsur--Levitzki Theorem ($k=2n$) and
therefore is sharp. 

Proposition~\ref{st} is sharp for $n=1$, and Corollary~\ref{6} is sharp for $n=1$ or $n= 2$. More generally, we have

\begin{Pro}  The standard identity of degree $k=2(n+\lfloor m/2\rfloor)-1$ does not hold in $M_nE^m$ if the base ring $R$ is  a field of characteristic either zero or a prime  $p> 2\lfloor m/2\rfloor$.
\end{Pro}

\begin{proof} Consider the $2n-1$ matrices $e_{12}$, $e_{23}$, \dots, $e_{n-1,n}$, $e_{nn}$, $e_{n, n-1}$, $e_{n-1,n-2}$, \dots, $e_{21}$, together with the $2\lfloor m/2\rfloor$ further matrices $v_ie_{11}$, where $i=1,\dots, 2\lfloor m/2\rfloor$.   The standard polynomial $s_k$ evaluated at these $k$ matrices is the same as $$s_{2\lfloor m/2\rfloor+1}(e_{11}, v_1e_{11}, \dots, v_{2\lfloor m/2\rfloor}e_{11}).$$ By Lemma~\ref{perm}(b), applied to the trivial partition, this is $$(2\lfloor m/2\rfloor)!v_1\cdots v_{2\lfloor m/2\rfloor}e_{11}\ne 0.$$
\end{proof}

\begin{?} Does the standard identity of degree $2(n+\lfloor m/2\rfloor)$ hold in $M_nE^m$?
\end{?}
 For $m=0$, or $m=1$, or $n=1$, or $n=2$, the answer is clearly affirmative.
%\noindent
\section*{Acknowledgements}
I am grateful to the unnamed referee of this paper and to M\'aty\'as Domokos  for  useful  comments.

\end{document}